\documentclass[12 pt]{amsart}

\usepackage{enumerate,amsmath,amsthm,amssymb,tabularx}
\usepackage{hyperref}
\usepackage[utf8]{inputenc}
\usepackage[T1]{fontenc}
\usepackage{tikz}

\usepackage[margin = 1.25in]{geometry}
\usepackage{mathtools}
\usepackage{diagbox}
\usepackage{verbatim}

\usepackage{crimson}

\theoremstyle{plain}
\newtheorem{proposition}{Proposition}

\newtheorem{lemma}{Lemma}
\newtheorem{theorem}{Theorem}

\theoremstyle{definition}
\newtheorem{definition}{Definition}

\theoremstyle{remark}

\author[C. Corona]{Connie Corona}
\address{CC: Department of Mathematics \\ California State University San Bernardino \\ San Bernardino,
  CA 92407, USA} 
\email{conniec@coyote.csusb.edu}
\author[Z. Hasan]{Zahid Hasan}
\address{ZH: Department of Mathematics \\ California State University San Bernardino \\ San Bernardino,
  CA 92407, USA} 
\email{zhasan@csusb.edu}
\author[B. Lim]{Bronson Lim}
\address{BL: Department of Mathematics \\ California State University San Bernardino \\ San Bernardino,
  CA 92407, USA} 
\email{BLim@csusb.edu}

\begin{document}

\title{Symmetric Generation of \(J_2\) on 32 Letters}


\begin{abstract}
  We give a computer-free proof that \(J_2\) is isomorphic to the progenitor \(2^{\star 32}: (2^{1+4}:A_5)\) factored by two relations, one of length 3 and and one of length 6, in the symmetric generators.
\end{abstract}

\maketitle

\section{Introduction}
\label{sec:intro}

\subsection{The Sporadic Simple Group \(J_2\)}

The second Janko group \(J_2\) is one of twenty six sporadic simple groups and
is of order \(604800\). Originally conjectured to exist by Janko \cite{janko} as a simple group with an involution centralized by \(2^{1+4}:A_5\). It was constructed explicitly by Hall and Wales as a
rank 3 permutation group on 100 letters, see \cite{hall-wales-69}. There have
been other constructions since then, see \cite{rahimipour-hossein-21} and
\cite{wilson-86}. 

In this paper we give an alternative construction using techniques of
symmetric generation. The automorphism group of \(J_2\colon 2\) was constructed by Curtis in \cite[Thm 5.4]{curtis} as a homomorphic image of the progenitor \(2^{\star 36}\colon (U_3(3)\colon 2)\). However, this is the first construction of \(J_2\) using symmetric generation. The progenitor used is motivated by the original conjecture of Janko.

\subsection{Symmetric Generation of Simple Groups}

All non-Abelian simple groups contain a set of generating involutions such that
the normalizer acts transitively on the generating set. These highly symmetric
generating sets yield presentations of such groups as quotients of subgroups of
\(2^{\star n}:\Sigma_n\). We can use this presentation to efficiently compute
products and inverses in these groups, see \cite[Appendix 2]{curtis}.

At the time of this writing, symmetric generating sets have yet to be discoverd
for the monster \(\mathbb{M}\), the baby monster \(\mathbb{B}\), and the
Thompson group \(\mathrm{Th}\). There are conjectures in this area, see \cite[Appendix A]{fairbairn-thesis},
but technology is not currently equipped to deal with such large groups yet.

Our presentation arises from a generating set of size 32 with normalizer
isomorphic to \(2^{1+4}:A_5\). The 32 point action is given by The coset action on the subgroup \(2^{1+4}\). The presentation will be similar to \cite{wiedorn} and \cite{L-Hasan}.

\subsection{Outline of Paper}

In \S \ref{sec:prelims} we collect the relevant preliminary materials on symmetric generation, provide a proof that \(J_2\) is symmetrically generated by a generating set of size \(32\) with normalizers isomorphic to \(2^{1+4}:A_5\), and define our construction of \(J_2\). In \S \ref{sec:dce} we use the techniques of double coset enumeration to construct a Cayley table for the action of our group. In \S \ref{sec:main} we use the Cayley table to prove that our group is isomorphic to \(J_2\).

\subsection{Acknowledgements}

The last author is grateful for the first two authors for including him in this project.

\section{Preliminaries on Symmetric Generation and \(J_2\)}
\label{sec:prelims}

Throughout the paper, we will use standard ATLAS notation for finite groups and
related concepts as described in \cite{ATLAS}.

\subsection{Progenitors}

This material is a review of the techniques described in \cite{curtis}.
Denote by \(2^{\star n}\) the \(n\)th free product of the cyclic group \(C_2\)
with itself. We set \(t_i\) to be the nontrivial element in the \(i\)th copy of
\(C_2\). Then \(2^{\star n} = \langle t_1,\ldots,t_n\mid t_i^2 = 1\text{ for }i
=1,\ldots,n\rangle\). There is a natural embedding of the symmetric group on
\(n\) letters, \(\Sigma_n\), in \(\mathrm{Aut}(2^{\star n})\) acting
transitively on the generators.

Suppose \(\mathcal{N}\subset \Sigma_n\) is a transitive subgroup, which we call
the \textit{control group}. Define \(\mathcal{P}\) to be the corresponding
split extension

\[
	\mathcal{P} = 2^{\star n}:\mathcal{N}.
\]

Then \(\mathcal{P}\) is called a \textit{progenitor} and the elements of
\(t_i\) are called \textit{symmetric generators}, \cite{curtis}. Any
element of \(\mathcal{P}\) is of the form \(\pi\omega\) where
\(\pi\in\mathcal{N}\) and \(\omega\) is a word on the symmetric generators.
There is a distinguished copy of \(\mathcal{N}\) inside \(\mathcal{P}\) by
taking \(\omega\) to be the empty word. We will also refer to this copy of
\(\mathcal{N}\) as the control group. 

\begin{definition}

    An epimorphic image of the progenitor, \(\varphi\colon \mathcal{P\to G}\)
    is called \textbf{symmetrically generated} if the following conditions are
    satisfied:

	\begin{enumerate}[(a)]

		\item the restriction to the control group is an isomorphism onto its image;

        \item the images of the symmetric generators,
            \(\varphi(t_1),\ldots,\varphi(t_n)\), are distinct involutions;
            
        \item the images of the symmetric geneators generate \(\mathcal{G}\),
            i.e. \(\mathcal{G} = \langle \varphi(t_i)\rangle_{i=1,\ldots,n}\).

	\end{enumerate}

\end{definition}

If \(\mathcal{G}\) is symmetrically generated we abuse notation and write
\(t_i\) for the image of \(t_i\) in \(\mathcal{G}\). It will be clear from
context where \(t_i\) lies.

Symmetrically generated groups arise, in practice, by quotienting a progenitor
by the normal closure of a finite set of elements. That is, if
\(\pi_1,\ldots,\pi_m\in\mathcal{N}\) and \(\omega_1,\ldots,\omega_m\) are words
in the symmetric generators, we define

\[
	\mathcal{G} = \frac{\mathcal{P}}{\pi_1\omega_1,\ldots,\pi_m\omega_m}
\]
to mean the quotient of \(\mathcal{P}\) by the normal closure of the group
generated by \(\pi_1\omega_1,\ldots,\pi_m\omega_m\) in \(\mathcal{P}\). The
elements \(\pi_1\omega_1,\ldots,\pi_m\omega_m\) are refered to as relations. 

A general set of relations need not determine a symmetrically generated group or
even a finite one. An interesting question is to determine which finite groups
are symmetrically generated. As a corollary of the Feit-Thompson odd order
theorem, any finite non-Abelian simple group is symmetrically generated
\cite[Lemma 3.6]{curtis}. In particular, the sporadic group \(J_2\) is
symmetrically generated.

\subsection{Symmetric Generation of \(J_2\)}

Our main theorem is motivated by the following proposition. \(J_2\) has a maximal subgroup \(2^{1+4}:A_5\). The where we look inside
the group structure of \(J_2\) for a symmetric generating set consisting of
involutions with control group \(2^{1+4}:A_5\).

\begin{proposition}

    There exists a symmetric generating set \(T = \{t_1,\ldots,t_{32}\}\) for
    \(J_2\) with control group \(2^{1+4}:A_5\).

    \label{prop:sym-gen-exist}
\end{proposition}

\begin{proof}

    Consider the maximal subgroup \(\mathcal{N} = 2^{1+4}:A_5\) inside of \(J_2\).
    Take an element \(t\) of order \(2\) in the centralizer of \(A_5\) which is
    not the center of \(2^{1+4}:A_5\). Then the number of conjugates of \(t\) under
    \(\mathcal{N}\) are \(|\mathcal{N}:\mathcal{N}^t| = |\mathcal{N}:A_5| =
    32\). We can label these conjugates as \(t_1,\ldots,t_{32}\) so that the
    generators of \(\mathcal{N}\) act via

    \begin{align*}
        x=&(1, 2)(3, 5, 7, 11, 17, 4, 6, 9, 14, 22)(8, 13, 20, 29, 23, 10, 16, 25, 30,18) \\
        &(12, 19, 28, 32, 26, 15, 24, 27, 31, 21) \\
        y=&(1, 3)(2, 4)(5, 8)(6, 10)(7, 12, 17, 27, 14, 23) \\
        &(9, 15, 22, 28, 11, 18)(13, 21, 24, 30, 32, 20)(16, 26, 19, 29, 31, 25).
    \end{align*}
	
    We now see that the \(t_i\) generate \(J_2\). Define \(H = \langle
    t_1,\ldots,t_{32}\rangle\). Since \(\mathcal{N}\) permutes the \(t_i\) we
    know \(\mathcal{N}\leq N_{J_2}(H)\). Since \(t\notin \mathcal{N}\) we have
    that \(\mathcal{N}\) is a proper subgroup of the normalizer. This implies
    that \(H = J_2\).
    
\end{proof}

The centralizer of \(A_5\) in \(J_2\) is \(D_{10}\) so we also know there is a symmetric generating set where the \(t_i\) have order \(5\). See \cite[Chapter 6]{curtis} for more general progenitors than those generated by involutions, such as those generated by cyclic groups of other orders or monomial progenitors.

\subsection{A Progenitor Corresponding to \(2^{1+4}:A_5\)}

Let \(x,y\) be defined as in Proposition \ref{prop:sym-gen-exist}. We take the action of \(2^{1+4}:A_5\) on the set of cosets of \(A_5\) then \(2^{1+4}:A_5\cong\langle x,y\rangle\). We define the progenitor \(\mathcal{P}\) to be

\[
	\mathcal{P} = 2^{\star 32}:(2^{1+4}:A_5).
\]

Now set \(\pi = xy^{-2}xy\) and \(\tau = x^5y^3\). Explicilty, we have
\begin{align*}
	\pi = &(1, 3, 12, 16, 7)(2, 4, 15, 13, 9)(5, 28, 14, 23, 26) \\
	&(6, 27, 11, 18, 21)(8, 19, 30, 20, 22)(10, 24, 29, 25, 17) \\
	\tau = &(1, 4)(2, 3)(5, 10)(6, 8)(7, 28)(9, 27)(11, 12)(13, 29)\\
	&(14, 15)(16, 30)(17, 18)(19, 20)(21,31)(22, 23)(24, 25)(26, 32)
\end{align*}

Our main object of study is the homomorphic image:
\begin{align*}
	\mathcal{G} = &\frac{\mathcal{P}}{(\tau t_1)^3, (\pi t_2)^6} \\
	= &\langle x,y,t|x^{10}, y^6, (xy^{-2}x)^2, (xy^2x^2)^2, (y^{-1}x^{-1})^5, (xy^2x^{-1}y^{-1})^2, x^{-1}y^{-1}x^5yx^{-4}, \\
	& yx^{-2}y^{-1}x^3yxy^3x^{-1},
t^2, (t,x^2), (t,y^2),
(xy^{-2}xyt^x)^6, (xt)^5\rangle.
\end{align*}

\begin{theorem}\label{thm:main-result}

	The group \(\mathcal{G}\) is isomorphic to the Janko group \(J_2\).

\end{theorem}

The proof of this theorem will occupy the next two sections. We also have that \(J_2\) has a symmetric generating set of size 32 consisting of symmetric generators of order \(5\). Thus it is also a true image of the permutation progenitor \(5^{\star 32}\colon (2^{1+4}: A_5)\). In particular, we have tabulated:

\begin{theorem}\label{thm:additional-result}
	The Janko group \(J_2\) is isomorphic to
	\begin{align*}
		\mathcal{G}' = &\frac{5^{*32}:(2^{1+4}:A_5)}{((xy)^2t^{x^{-2}y^{-1}x^2yx})^7, (y^3t^{xy^{-2}x^{-2}y^{-1}})^3} \\
		= &\langle x,y,t|x^{10} ,y^6 , (x  y^{-2 } x)^2 , (x  y^2  x^2)^2 ,(y^{-1} x^{-1})^5 , (x  y^2  x^{-1}  y^{-1})^2 , x^{-1} y^{-1}  x^5  y  x^{-4} , \\
		&y  x^{-2} y^{-1} x^3  y  x  y^3  x^{-1}, t^5, (t,(x^2)),(t,(y^2)), ((xy)^2t^{x^{-2}y^{-1}x^2yx})^7, (y^3t^{xy^{-2}x^{-2}y^{-1}})^3\rangle.
	\end{align*}
\end{theorem}

Theorems \ref{thm:main-result} and \ref{thm:additional-result} are readily verified using MAGMA, \cite{MAGMA}. Additionally, MAGMA was used for verification of several relations.

\section{Double Coset Enumeration over \(2^{1+4}:A_5\)}\label{sec:dce}
\subsection{Double Coset Enumeration}

We use double coset enumeration techniques to establish a primitive action of
\(J_2\) on the set of left cosets. We then appeal to Iwasawa's simplicity
criterion. The remainder of this section is devoted to the former, i.e.
computing the Cayley Graph for the action of \(\mathcal{G}\) on single cosets
of the form \(\mathcal{N}\omega\) where \(\mathcal{N} = 2^{1+4}:A_5\) is the
control group.

If \(\omega_1,\omega_2\) are words in the symmetric generators, we will write \(\omega_1\sim\omega_2\) whenever \(\mathcal{N}\omega_1=\mathcal{N}\omega_2\).

\subsection{Relations}

We prove necessary relations to perform the double coset enumeration. We have the following immediate relations
\begin{equation}\label{reln:given-3}
	\tau = t_1t_4t_1
\end{equation}
and
\begin{equation}\label{reln:given-6}
	\pi^4 = t_2t_9t_{13}t_{15}t_4t_2 
\end{equation}
given by expanding the relations \( (\pi t_2)^6 = (\tau t_1)^3 = e\).

We devote the rest of this subsection to the proof of several lemmas that are necessary to complete the double coset enumeration process.

\begin{lemma}
	We have the relation:
	\begin{equation}\label{reln:182329-word}
	t_1t_8t_{23}t_{29} = x^4yxyx^{-1}t_{18}t_{12}t_{20}t_{30}
	\end{equation}
\end{lemma}

\begin{proof}
We compute:
\begin{align*}
	t_1t_8t_{23}t_{29} &= yx^2y^{-2}x^{-2}t_6t_9t_{23}t_{13}t_6(t_1t_3)t_{29} & (\ref{reln:given-6}) \\
	&= (yxy^{-1})^4t_{12}t_2t_{30}t_{11}t_{12}t_1t_7t_{21}(t_{10}t_{29}) & (\ref{reln:given-6}) \\
	&= x^2yx^2y^{-1}x^{-1}y^{-1}t_{19}t_7t_8t_{25}t_{19}(t_9t_2t_{22})t_{10} & (\ref{reln:given-3}) \\
	&= xyx^4y^{-1}t_{15}t_{30}t_{32}t_{21}t_{15}t_9t_{27}(t_{29}t_{10}) & (\ref{reln:given-6}) \\
	&= yxyx^{-1}y^{-1}xyx^{-1}t_{24}t_8t_{23}t_{22}t_{24}t_1t_3t_{29} & (\ref{reln:given-3}) \\
	&= yxyx^{-1}y^{-1}xyx^{-1}t_{24}t_8, (t_{23}t_{22}t_{24})t_1t_3t_{29} & (\ref{reln:given-3}) \\
	&= x^4yx^{-1}y t_{28}t_8(t_{23}t_{31})t_{28}t_{1}t_3t_{29} & (\ref{reln:given-6}) \\
	&= yxy^{-2}x t_7t_{22}t_{23}t_{25} (t_7t_1t_3)t_{29} & (\ref{reln:given-6}) \\
	&= (yx yx^{-1}y^{-1})^2 t_1t_{28}t_{30} t_{17}t_7 t_{21}(t_{10}t_{29}) & (\ref{reln:given-6}) \\
	&= y^2x^{-1}y^{-1}x^2t_9(t_4t_8 t_{26})t_2t_{22}t_{10} & (\ref{reln:given-3}) \\
	&= xy^2x^{-1}y^{-2}x t_2(t_4t_2 t_9 )t_2t_{22}t_{10} & (\ref{reln:given-6}) \\	
	&= y^2xy^3x^{-1}yt_{32} t_4t_2(t_{11}t_2t_{22})t_{10} & (\ref{reln:given-6}) \\
	&= yx^{-1}y^2xt_{17} t_7 t_{16} t_{11} (t_2 t_5t_{10}) & (\ref{reln:given-6}) \\
	&= (yxy^{-2})^2t_{32}t_{24}t_{20}t_2t_2t_{11}t_{30} & (\ref{reln:given-6}) \\
	&= x^4yxyx^{-1} t_{18}t_{12}(t_{20}t_{11})t_{30} & (\ref{reln:given-3}) \\
	&= x^4yxyx^{-1}t_{18}t_{12}t_{20}t_{30}.
\end{align*}
\end{proof}

\begin{lemma}
	We have the relation
	\begin{equation}\label{reln:182329}
		t_1t_8t_{23}t_{29} = y^{-1}x^{-3}
	\end{equation}
\end{lemma}

\begin{proof}
We compute:
\begin{align*}
	t_1t_8t_{23}t_{29} &= x^4yxyx^{-1}t_{18}t_{12}t_{20}t_{30} & (\ref{reln:182329-word}) \\
	t_{29}t_{23}t_8t_1 &= xy^{-1}x^{-1}y^{-1}xt_8t_{14}t_{23}t_1 & (\text{inverting})\\
	t_{29} t_{23} t_8 t_1 &= xyxy^2xyxt_8t_{14}t_{23}t_1\\
	t_{29} t_{23} t_8 &= xyxy^2xyx t_8 t_{14}t_{23}\\
	x^4y xyx^{-1}  t_{29}t_{23}t_8 &= t_8 t_{14} t_{23}. \\
\end{align*}
Now we conjugate by \(y^2xy^{-1}xy^{-1}x\):
\begin{align*}
	x^{-3}y^{-1}x^{-1} yx^{-1} t_4 t_{14} t_{28}  &= t_{28} t_{18} t_{14} \\
	x^{-3}y^{-1}x^{-1} yx^{-1} t_4 (t_{14} t_{28}t_{14})t_{18}t_{28}  &= e & (\ref{reln:182329-word}) \\
	x^{-2}y^{-1} x^{-1}y^{-1}x^2t_{25}t_{32}t_{19}(t_{24}t_{18}t_{28}) &= e & (\ref{reln:given-6}) \\
	yxy^{-1}x^{-2}y^{-1} t_{30}  t_{20}  t_{28}  t_{28} t_{12} t_{18} &= e & (\ref{reln:given-6}) \\
	yxy^{-1}x^{-2}y^{-1} t_{30} t_{20}t_{12}t_{18} &= e. \\
\end{align*}
This yields
\begin{align*}
	t_1t_8t_{23}t_{29} &= x^4yxyx^{-1}t_{18}t_{12}t_{20}t_{30} & (\ref{reln:182329-word}) \\
	&= x^4yxyx^{-1}y x y^{-1}x^{-2}y^{-1} \\
	&= y^{-1}x^{-3}.
\end{align*}
\end{proof}

\begin{lemma}
	We have the relation
	\begin{equation}\label{reln:231823}
		t_{23}t_{18}t_{23} = x^2y^{-1}x^{-1}y^{-2}x^{-1}t_{23}t_3t_{27}.
	\end{equation}
\end{lemma}

\begin{proof}
We first show that \(t_{28}t_{27}t_{23} = xy^2xyx^{-2}t_{23}t_{27}t_3\):
\begin{align*}
t_1(t_2t_8) &= xy^2xy^{-1}x^{-1}(t_{22}t_{25})t_{16} & (\ref{reln:182329}) \\
&= xyx^2yxyt_7(t_6t_{16}) & (\ref{reln:182329}) \\
&= x^{-1}y^3xyt_1t_6. & (\ref{reln:given-3})
\end{align*}
Thus \(t_1t_2t_8 = x^{-1}y^3xyt_1t_6\). We compute:
\begin{align*}
	(t_1t_8)t_6 &= y^3x^3 (t_{26}t_{27})t_6 & (\ref{reln:182329}) \\
	&= x^{-1}yxyx^{-1}yx(t_{10}t_{26})t_3 & (\ref{reln:182329}) \\
	&= y^2t_{10}t_3. & (\ref{reln:given-3})
\end{align*}
Thus \(t_1t_8t_6 = y^2t_{10}t_3\). Conjugating \(t_1t_2t_8 = x^{-1}y^3xyt_1t_6\) by \(xy^{-2}x^{-1}y^{-1}x^{-1}y\) and \(t_1t_8t_6 = y^2t_{10}t_3\) by \(x^{-1}y^{-1}x^{-1}y^3\) we obtain
\[
t_{28}t_{27}t_{23} = x^2yxy^2t_{28}t_{31}\text{ and }t_{23}t_{27}t_3 = yx^{-4}yt_{28}t_{31}.
\]
Hence, \(t_{28}t_{27}t_{23} = xy^2xyx^{-2}t_{23}t_{27}t_3\) and so
\[
t_{23}t_3t_{27} = xy^2xyx^{-2}t_{23}t_{18}t_{23}
\]
as desired.
\end{proof}

\begin{lemma}
We have the relation
	\begin{equation}\label{reln:22224}
		t_2t_{22}t_{24} = x^{-1}y^2xyx^{-1}yt_2t_{22}t_{21}.
	\end{equation}
\end{lemma}

\begin{proof}
We compute
\begin{align*}
t_2t_{22}(t_{24}) &= (xyx^{-1})^3 t_{22}t_2t_{24}(t_5) & (\ref{reln:given-3}) \\
&= (x^{-1}y^{-1}x^{-1})^2 & (\ref{reln:given-3}) \\
&= x^2yxyx^{-1}yt_7t_{15}(t_8t_5)t_1 & (\ref{reln:182329}) \\
&= y^{-2}x^{-1}y^{-1}x^{-2}y^{-1} (t_{27}t_{11}t_8)t_1 & (\ref{reln:given-3}) \\
&= xy^{-1}xyxy^2t_2t_1.
\end{align*}
Thus \(t_2t_{22}t_{24} = xy^{-1}xyxy^2t_2t_1\). Conjugating this relation by \( (x^2y^{-1}x^{-2})^2\) yields
\[
t_2t_{22}t_{21} = y^{-1}x^{-1}yx^{-1}y^{-1}t_2t_1
\]
Now
\begin{align*}
t_2t_{22}t_{24} &= xy^{-1}xyxy^2(t_2t_1) \\
&= xy^{-1}xyxy^2(y^{-1}x^{-1}yx^{-1}y^{-1})^{-1}t_2t_{22}t_{21}  \\
&= x^{-1}y^2xyx^{-1}yt_2t_{22}t_{21}.
\end{align*}
\end{proof}

\begin{lemma}
We have the relation
	\begin{equation}\label{reln:2125}
		t_2t_1t_2t_5 = t_5t_2t_1t_2.
	\end{equation}
\end{lemma}

\begin{proof}
We compute
\begin{align*}
	t_6t_5t_6(t_1) &= xyx^{-1}yx^{-1}y^{-2}(t_{23}t_{18}t_{23})t_{27}t_6t_{25} & (\ref{reln:182329}) \\
	&= yxy^2xyxyt_{23}t_3(t_6t_{25}) & (\ref{reln:231823}) \\
	&=x^2y^{-1}x^{-1}y^{-1}(t_2t_{22}t_{24})t_{10} & (\ref{reln:182329}) \\
	&= xy^{-1}x^{-4}yt_2t_{22}t_{21}t_{10} & (\ref{reln:22224}) \\
	&= (x^{-1}yx)^3t_2t_5. & (\ref{reln:given-6})
\end{align*}
Thus \(t_6t_5t_6t_1 = (x^{-1}yx)^3t_2t_5\) and by inverting:
\begin{align*}
t_1t_6t_5t_6 &= (x^{-1}yx)^{-3}t_2t_5 = t_6t_5t_6t_1.
\end{align*}
\end{proof}

\begin{lemma}\label{lem:12125}
	We have the relation
	\begin{equation}\label{reln:12125}
		t_1t_2t_1t_2 = (yx^{-1}yxy^{-1})^3t_1t_2t_1t_2t_5
	\end{equation}
\end{lemma}

\begin{proof}
	Now
	\begin{align*}
	t_1(t_2t_1t_2t_5) &= (t_1t_5)t_2t_1t_2 & (\ref{reln:2125}) \\
	&= (yx^{-1}yxy^{-1})^3t_1t_2t_1t_2. & (\ref{reln:given-3})
	\end{align*}
\end{proof}

\begin{lemma}
	We have the relation:
	\begin{equation}\label{reln:97927}
	t_1t_2t_1t_2 = (y^{-1}xyx^{-1}y)^3t_9t_7t_9t_2t_7
	\end{equation}
\end{lemma}

\begin{proof}
	We compute:
	\begin{align*}
		t_1t_2t_1t_2 &= t_1t_2t_1(t_2t_7)t_7 \\
		&= (y^{-1}xyx^{-1}y)^3t_9t_7t_9t_2t_7.
	\end{align*}
\end{proof}

\begin{lemma}
	We have the relation:
	\begin{equation}\label{reln:2223107}
		t_1t_2t_1t_2 = x^{-1}y^{-1}x^{-1}y^3x^{-1}t_{22}t_{23}t_{10}t_7
	\end{equation}
\end{lemma}

\begin{proof}
	We compute:
	\begin{align*}
		t_1t_2t_1(t_2) &= (y^{-1}xyx^{-1}y)^3t_9t_7(t_9t_2)t_7 & (\ref{reln:given-6}) \\
		&= x^{-1}y^{-2}x^{-1}y^2x t_2t_1t_9(t_{27}t_{19}t_5)t_7 & (\ref{reln:given-3}) \\
		&= x^{-2}y^{-1}x^{-1}y^{-1}xy t_{30}(t_{29}t_{17})t_{27}t_3t_{10}t_7 & (\ref{reln:given-3}) \\
		&= x^{-1}y^{-2}xyx^{-1} (t_{22} t_{29} t_{27} t_3) t_{10} t_7 & (\ref{reln:given-3}) \\
		&= x^{-1}y^{-1}x^{-1}y^3x^{-1}t_{22}t_{23}t_{10}t_7.
	\end{align*}
\end{proof}

\begin{lemma}
We have the relation:
\begin{equation}\label{reln:2510}
t_2t_5t_{10} = x^2yxyxy^{-1}t_2t_5t_{31} = x^2yxyxy^{-1}xyxy^2t_2t_{22}t_{23}.
\end{equation}
\end{lemma}

\begin{proof}
We compute
\begin{align*}
t_2t_5t_{10}t_{31}t_5t_2 &= t_2t_5t_{10}t_{31}t_5(t_2t_7)t_7 \\
&=(y^{-1}xy x^{-1}y)^3 (t_7t_{16}t_{29})t_{18}t_{16}t_2t_7 & (\ref{reln:given-3}) \\
&=xyx^{-1}yx^2y^{-1} t_7t_1t_{14}t_{18}t_{16}(t_2t_7) & (\ref{reln:given-6}) \\
&=x^3y^3x t_2t_9t_{20}(t_{31}t_5t_2) & (\ref{reln:given-3}) \\
&= (xy^{-2})^2(t_4t_{18}t_{20})t_{31}t_{28}t_4 & (\ref{reln:given-6}) \\
&= x^2y^2  t_4t_8t_{26}t_{31}t_{28}t_4 & (\ref{reln:given-6}) \\
&= x^2yxyxy^{-1} & (\ref{reln:given-6})
\end{align*}
This gives the first relation. For the second we have by (\ref{reln:given-6})
\[
	t_2t_5t_{31}t_{23}t_{22}t_2 = xyxy^2
\]
which completes the proof.
\end{proof}

\begin{lemma}
	We have the relation:
	\begin{equation}\label{reln:12121}
		t_1t_2t_1t_2t_1 = x^5.
	\end{equation}
\end{lemma}

\begin{proof}
	Notice that
	\[
	(y^{-1}xy^{-1}x^{-1}y^{-1}x t_{22}t_{23}t_{10}t_7t_5t_1)^{-1} = (y^{-1}xy^{-1}x^{-1}y^{-1}x)^{-1}t_7t_{27}t_{25}t_{32}t_{22}t_2.
	\]
	We compute:
	\begin{align*}
	(t_1t_2t_1t_2)t_1 &= yx^{-1}yxy^{-1} (t_1t_2t_1t_2)t_5t_1 & (\ref{reln:12125}) \\
	&= y^{-1}xy^{-1}x^{-1}y^{-1}xt_{22}t_{23}t_{10}t_7t_5t_1 & (\ref{reln:2223107})\\
	&= (y^{-1}xy^{-1}x^{-1}y^{-1}x)^{-1}(t_7t_{27})t_{25}t_{32}t_{22}t_2 & (t_1t_2t_1t_2t_1 \text{ is an involution}) \\
	&= yx^2yx^{-1}yx(t_7t_{25})t_{32}t_{22}t_2 & (\ref{reln:given-3}) \\
	&= y^{-1}xy^{-2}xy^{-1}  t_7t_{21}t_{22}(t_{23}t_{32})t_{22}t_2 & (\ref{reln:given-6}) \\
	&= yx^{-1}y^{-1}xy^{-2}(t_2t_{22}t_{21}) t_{23}t_{22}t_2 & (\ref{reln:given-3}) \\
	&= xy^3xy^{-1}xt_2t_5t_{10}(t_{23}t_{22}t_2) & (\ref{reln:given-6}) \\
	&= y^{_2}xyx^{-2}yt_2t_5t_{10}t_{31}t_5t_2 & (\ref{reln:given-3}) \\
	&= y^{-1}xyxyx^{-1} (t_7t_{16}t_{29})t_{18}t_{16}t_2 & (\ref{reln:given-3}) \\
	&= x * y^{-1}xyx^{-1}y^{-1}x^2t_7t_1t_{14} t_{18}t_{16} (t_2t_7) & (\ref{reln:given-3}) \\
	&= x^{-1}yxy^{-1}x^{-1}y^{-2}t_2t_9t_{20} (t_{31}t_5t_2) & (\ref{reln:given-3}) \\
	&= yx^{-1}yx^{-1}y^{-1}x^{-1}(t_4t_{18}t_{20})t_{31}t_{28}t_4 & (\ref{reln:given-3}) \\
	&= x^{-2}yx^{-1}y^{-1}x(t_4t_8t_{26}t_{31}t_{28}t_4) & (\ref{reln:given-6})\\
	&= x^{-2}yx^{-1}y^{-1}xy^{-1}xy xy^{-1}& (\ref{reln:given-3}) \\
	&= x^5
	\end{align*}
\end{proof}

\subsection{Double Cosets of Length 1}

There is a single double coset of length 0, denoted \([\star] = \mathcal{N}e\mathcal{N}\). The action of \(\mathcal{N}\) is transitive on \(X\) so there is one double coset of length 1. We denote it \([t_1] = \mathcal{N}t_1\mathcal{N}\).

Since our relation is an equality between words of length 3, the point stabilizer \(\mathcal{N}^1\) and the coset stabilizer \(\mathcal{N}^{(1)}\) or equal. In particular, the orbits of \(\mathcal{N}^{(1)}\) on \(X\) are
\begin{align*}
	&\{ \{1\}, \{2\}, \{3,6, 7, 14, 17\}, \{4, 5, 9, 11, 22\}, \\
	&\{8,20,23,21,16,12,19,30,28,27,32,31,26,18,15,13,24,29,10,25\}\}
\end{align*}
We pick orbit representatives \(\{1,2,3,4,8\}\). Each of the cosets corresponding to \(\{2, 3, 8\}\) are new and we denote them by \([t_1t_2],\ [t_1t_3],\ [t_1t_8]\), respectively. We have \([t_1^2]=[\star]\). By Relation (\ref{reln:given-3}) we have \([t_1t_4]=[t_1]\).

\subsection{Double Cosets of Length 2}

\subsubsection{\([t_1t_2]\)}

The point stabilizer and coset stabilizer are equal: \(\mathcal{N}^{1,2} = \mathcal{N}^{(1,2)}\) and are both equal to the one point stabilizer \(\mathcal{N}^1\). In particular, the orbits are the same as the length 1 case. Picking the same orbit representatives \(\{1,2,3,4,8\}\) we have:

We have:
\begin{enumerate}
	\item[\(t_1\)] By Relation (\ref{reln:12121}) we have \([t_1t_2t_1]=[t_1t_2]\).
	\item[\(t_2\)] We have \([t_1t_2^2] = [t_1]\).
	\item[\(t_3\)] Conjugate Relation (\ref{reln:given-3}) by \(x^{-2}y^{-2}x\) to get
	\[
	x^{-2}y^{-2}xt_2 = t_2t_3
	\]
	and so we have 
	\[
		t_1(t_2t_3) \sim t_1(x^{-2}y^{-2}t_2) \sim t_4t_2 \sim (t_1t_3)^{x^{-2}y^{-2}xy}.
	\]
	Thus \([t_1t_2t_3] = [t_1t_3]\).
	\item[\(t_4\)] The coset \([t_1t_2t_4]\) is new.
	\item[\(t_8\)] Conjugating Relation (\ref{reln:182329}) by \(yxy^2x^{-1}yx\) and \(xyx^{-1}y^2x\) to get the relations 
	\begin{align*}
		t_2t_8 &= xy^2xy^{-1}x^{-1}t_{25}t_{16}
		t_{22}t_{25} &= y^{-1}x^{-2}y^{-1}t_7t_6
	\end{align*}
	and Relation (\ref{reln:given-3}) by \(x^{-1}yx^3yx^2\) to get
	\[
	(x^{-2}yx^2)^3t_6 = t_6t_{16}
	\]
	which gives us
	\begin{align*}
		t_1(t_2t_8) &\sim t_1xy^2xy^{-1}x^{-1}t_{25}t_{16} \sim (t_{22}t_{25})t_{16} \\
		&\sim t_7(t_6t_{16})\sim t_7(x^{-2}yx^2)^3t_6 \\
		&\sim t_1t_6 \sim (t_1t_3)^{x^{-4}y^2}.
	\end{align*}
	Thus \([t_1t_2t_8] = [t_1t_6] = [t_1t_3]\).
\end{enumerate}
There is one new double coset given by \([t_1t_2t_4]\).

\subsubsection{\([t_1t_3]\)}

The point stabilizer \(\mathcal{N}^{1,3}\) and the coset stabilizer \(\mathcal{N}^{(1,3)}\) are equal. The orbits on \(X\) are
\begin{align*}
	&\{\{1\},\{2\},\{3\},\{4\},\{5,22,11,9\},\{6,17,14,7\},\{8,18,15,28\}, \\
	&\{10,23,12,27\},\{13,24,26,32,30,29,20,19,31,25,21,16\}\}
\end{align*}
We take \(\{1, 2, 3, 4, 5, 6, 8, 10, 13\}\) as orbit representatives. We have:
\begin{enumerate}
	\item[\(t_1\)] By Relation (\ref{reln:12121}) we have \([t_1t_3t_1] = [t_1t_3]\).
	\item[\(t_2\)] Conjugate Relation (\ref{reln:given-3}) by \( x^{-2}y^{-2}x\) to get
	\[
	x^5y^3t_3t_2 = t_3
	\]
	and so we have
	\[
		t_1(t_3t_2)\sim t_1x^5y^3t_3 \sim t_4t_3 \sim (t_1t_2)^{x^3y^{-1}}.
	\]
	Thus \([t_1t_3t_2] = [t_1t_2]\).
	\item[\(t_3\)] We have \([t_1t_3^2] = [t_1]\).
	\item[\(t_4\)] We have
	\begin{align*}
		t_1t_3(t_4) &= xyx^3y^{-1}xt_{31}(t_{28}t_{16})t_{26}t_5 & (\ref{reln:182329}) \\
		&= x^{-1}yx^{-1}yx^{-1}y^{-1}(t_{21}t_{12})t_{10}t_{26}t_5 &  (\ref{reln:given-3}) \\
		&= xy^{-1}xyx^{-2}y^{-1}x^{-1}t_{21}(t_{10}t_{26})t_5 &  (\ref{reln:given-3}) \\
		&= yt_8t_{10}t_5 &  (\ref{reln:given-3}) \\
	\end{align*}
	But then \( t_8t_{10}t_5 \in [t_1t_2t_4]\). 
	\item[\(t_5\)] Conjugate Relation (\ref{reln:182329}) by \(y^{-1}, x^3y^{-1}\), and \(yx^{-2}yx^{-1}y^2\) to get the relations:
	\begin{align*}
	x^{-3}y^{-1}t_{19}t_{14} &= t_3t_5 \\
	y^3x^3t_{25}t_{31} &= t_4t_{19} \\
	xyx^{-1}y^{-2}xt_{26}t_2 &= t_{25}t_{31}
	\end{align*}
	Conjugate Relation (\ref{reln:given-3}) by \(xy^2x^2\) to get
	\[
	(yxyx^{-1}y^{-1})^3t_2 = t_2t_{14}.
	\]
	Then we have
	\begin{align*}
		t_1(t_3t_5) &\sim t_1x^{-3}y^{-1}t_{19}t_{14} \sim (t_4t_{19})t_{14} \\
		&\sim (t_{25}t_{31})t_{14} \sim t_{26}(t_2t_{14}) \\
		&\sim t_{26}(yxyx^{-1}y^{-1})^3t_2 \sim t_{10}t_2 \\
		&\sim (t_1t_8)^{x^{-1}yx^3yx}.
	\end{align*}
	Thus we have \([t_1t_3t_5] = [t_1t_8]\).
	\item[\(t_6\)] Conjugate Relation (\ref{reln:given-3}) by \((yx^{-1}yxy^{-1})^3\), \(xy^2xy^{-1}xy^{-1}\), and \(yx^{-3}y^2\) to get
	\begin{align*}
	(yx^{-1}yxy^{-1})^3t_6 &= t_6t_2 \\
	yx^{-3}y^2t_{12}t_4 &= t_{12} \\
	xy^2xy^{-1}xy^{-1}t_{10} &= t_{10}t_6. \\
	\end{align*}
	Then conjugate Relation (\ref{reln:182329}) by \(y^3xy^{-1}x^{-2}\) and \(y^{-2}\) to get
	\begin{align*}
	y^3xy^{-1}x^{-2}t_9t_{12} &= t_{10}t_2 \\
	y^{-2}t_8 &= t_{22}t_{25}t_{12}.
	\end{align*}
	Finally, we compute
	\begin{align*}
		t_1t_3t_6 &\sim t_1t_3(t_6t_2)t_2 \sim t_1t_3(yx^{-1}yxy^{-1})^3t_6t_2 \\
		&\sim t_5(t_{10}t_6)t_2 \sim t_5xy^2xy^{-1}xy^{-1}t_{10}t_2 \\
		&\sim t_8(t_{10}t_2)\sim t_8 y^3xy^{-1}x^{-2}t_9t_{12} \\
		&\sim t_{30}t_9(t_{12}) \sim t_{30}t_9 yx^{-3}y^2t_{12}t_4 \\
		&\sim (t_{22}t_{25}t_{12})t_4 \sim t_8t_4 \\
		&\sim (t_1t_2)^{x^{-4}}.
	\end{align*}
	Thus \([t_1t_3t_6] = [t_1t_2]\).
	\item[\(t_8\)] Conjugate Relation (\ref{reln:given-3}) by \( x^{-1}y^2x^3y^{-1}\) to get
	\[
		(x^{-1}yx)^3t_3t_8 = t_3.
	\]
	Then we have
	\[
		t_1(t_3t_8)\sim t_1(x^{-1}yx)^3t_3 \sim t_6t_3 \sim (t_1t_8)^{y^2x^{-1}y^3x}.
	\]
	Thus \([t_1t_3t_8] = [t_1t_8]\).
	\item[\(t_{10}\)] We have
	\begin{align*}
		t_1t_3(t_{10}) &= y^3t_3t_1t_{10} &  (\ref{reln:given-3}) \\
		&= yx^{-1}yx^{-1}y^{-2}x^{-1} t_{10}(t_5t_3)t_6t_2 & (\ref{reln:182329}) \\
		&= x^{-1}yx^2y^{-2}t_{11}t_{26}(t_{19}t_6)t_2 & (\ref{reln:182329}) \\
		&= yx^{-2}y^{-1}xy^{-1} (t_{29}t_7t_{19})t_2 &  (\ref{reln:given-3}) \\
		&= x^{-1}y^3xyt_1t_2 & (\ref{reln:182329}).
	\end{align*}
	Thus \([t_1t_3t_{10}] = [t_1t_2]\).
	\item[\(t_{13}\)] Conjugate Relation (\ref{reln:182329}) by \(x^{-2}y^3\) to get
	\[
		xy^{-2}x^{-1}yx^{-1}t_{30}t_{24} = t_3t_{13}
	\]
	Conjugate Relation (\ref{reln:given-3}) by \(x^3y^{-2}x^{-1}yx^{-1}\) to get
	\[
	x^2yx^2y^{-1}t_{22} = t_{22}t_{30}.
	\]
	Combining these yields
	\begin{align*}
	t_1(t_3t_{13}) &\sim t_1xy^{-2}x^{-1}yx^{-1}t_{30}t_{24}\sim (t_{22}t_{30})t_{24}\\
	&\sim t_{22}t_{24}\sim (t_1t_3)^{x^{-2}y^{-1}x^{-1}}.
	\end{align*}
	Thus \([t_1t_3t_{13}] = [t_1t_3]\).
\end{enumerate}

\subsubsection{\([t_1t_8]\)}

For this coset the coset stabilizer is larger than the point stabilizer as by (\ref{reln:182329}) we have \([t_1t_8] = [t_{29}t_{23}]\). Hence the element
\begin{align*}
	\gamma_{18}&=(1, 29, 3, 16)(2, 30, 4, 13)(5, 17, 8, 23)(6, 22, 10, 18) \\
	&(7, 14, 27, 12)(9, 11, 28, 15)(19, 26, 25, 31)(20, 32, 24, 21) 
\end{align*}
is in \(\mathcal{N}^{(1,8)}\). In particular we have
\[
\mathcal{N}^{(1,8)}\geq \langle \mathcal{N}^{1,8},\gamma_{1,8}\rangle\cong 2\cdot A_4.
\]

The action of \(\mathcal{N}^{(1,8)}\) on \(X\) has orbits
\begin{align*}
	&\{\{1,3,29,16,25,31,19,26\},\{2,4,30,13,20,32,24,21\}, \\
	&\{5,8,17,23,14,27,12,7\},\{6,10,22,18,11,28,15,9\}\}
\end{align*}
We pick the representatives \(\{1, 2, 6, 8\}\). Each of the orbits goes back to a word of length 2 or 1 as we will see:
\begin{enumerate}
	\item[\(t_1\)] Conjugate Relation (\ref{reln:given-3}) by \( (yx^{-1}yxy^{-1})^3\), \(xy^2xy^{-1}xy^{-1}\), and \(x^5y^3\) to get
	\begin{align*}
	(yx^{-1}yxy^{-1})^3t_1t_5 &= t_1 \\
	xy^2xy^{-1}xy^{-1}t_5 &= t_5t_8 \\
	x^5y^3t_6&=t_5t_1.
	\end{align*}
	We compute
	\begin{align*}
		(t_1t_8)t_1 &\sim t_1(t_5t_8)t_1 \sim t_1xy^2xy^{-1}xy^{-1}t_5t_1 \\
		&\sim t_3(t_5t_1) \sim t_3 x^5y^3t_6 \\
		&\sim t_{10}t_5 \sim (t_1t_3)^{x^3yxy}.
	\end{align*}
	Thus \([t_1t_8t_1] = [t_1t_3]\).
	\item[\(t_2\)] Conjugate Relation (\ref{reln:182329}) by \(y^3x^{-1}yx\) to get
	\[
	x^2yx^{-1}yt_{22}t_{32} = t_8t_2
	\]
	and Relation (\ref{reln:given-3}) by \(x^{-1}y^{-1}x^{-1}yx^{-1}y\) to get
	\[
	x^2yx^2y^{-2}t_{30} = t_{30}t_{22}.
	\]
	Then we have
	\begin{align*}
		t_1(t_8t_2) &\sim t_1x^2yx^{-1}yt_{22}t_{32} \sim (t_{30}t_{22})t_{32} \\
		&\sim t_{30}t_{32} \sim (t_1t_8)^{xyxy^{-1}xy^{-1}x}.
	\end{align*}
	Thus \([t_1t_8t_2] = [t_1t_8]\).
	\item[\(t_6\)] Conjugate Relation (\ref{reln:182329}) by \(xy^{-2}x^{-1}yx^{-2}y\) and by \(xyxy^2\) to get
	\begin{align*}
		y^3x^3t_{26}t_{27} &=t_1t_8 \\
		(yxy^3x^2)^{-1}t_{26}t_3 &= t_{27}t_6.
	\end{align*}
	Conjugate Relation (\ref{reln:given-3}) by \(yxyx^{-2}yx\) to get
	\[
	(yxyx^{-1}y^{-1})^3t_{10} = t_{10}t_{26}.
	\]
	We compute
	\begin{align*}
		(t_1t_8)t_6 &\sim (t_{26}t_{27})t_6 \sim (t_{10}t_{26})t_3 \\
		&\sim t_{10}t_3 \sim (t_1t_3)^{x^{-1}yx^3yx}.
	\end{align*}
	\item[\(t_8\)] We have \([t_1t_8^2] = [t_1]\).
\end{enumerate}

\subsection{The Double Coset of Length 3}

The coset stabilizer is larger than the point stabilizer in this case:

\begin{lemma}\label{lem:N124}
	We have
	\[
	\mathcal{N}^{(1,2,4)}\geq \langle \mathcal{N}^{1,2,4}, \rho\rangle
	\]
	where
	\begin{align*}
		\rho & =(1, 24, 3, 20)(2, 19, 4, 25)(5, 15, 8, 11)(6, 12, 10, 14) \\
		&=(7, 22, 27, 18)(9, 17, 28, 23)(13, 26, 30, 31)(16, 21, 29, 32) 
	\end{align*}
\end{lemma}

\begin{proof}
We show \([t_1t_2t_4] = [t_{24}t_{19}t_{25}]\):
\begin{align*}
t_1t_2(t_4) &= y^3xy^{-2}t_5(t_6t_{19})t_{32}t_{25} & (\ref{reln:182329}) \\
&= y^{-1}xy^{-1}x^{-1}y^2x t_{24}t_6(t_{32})t_{25} &  (\ref{reln:given-3}) \\
&= y^{-1}xyx^{-2} yx t_{26}(t_{31}t_{15})t_1t_{19}t_{25} & (\ref{reln:182329}) \\
&= yx^{-3}y^{-1}x^{-1}y t_{32}t_{19}(t_9t_1)t_{19}t_{25} & (\ref{reln:182329}) \\
&= yx^{-1}y^{-2}xy (t_{23}t_{12})t_9t_{19}t_{25} &  (\ref{reln:given-3}) \\
&= y^{-1}x^3yx t_{28}(t_{25}t_9)t_{19}t_{25} & (\ref{reln:182329}) \\
&= xyxyx^2y^2 t_{24}t_{25}t_{19}t_{25} & (\ref{reln:182329}) \\
&= xyxyx^{-3}y^{-1} t_{24}t_{19}t_{25} &  (\ref{reln:given-3}) \\
\end{align*}
Thus the element \(\rho\) stabilizes the coset \([t_1t_2t_4]\). Since \([t_{24}t_{19}t_{25} = t_3t_4t_2]\) we have
\[
\mathcal{N}^{(1,2,4)}\geq  \langle\mathcal{N}^{1,2,4},\rho\rangle\cong 2^4:A_4.
\]
\end{proof}

By Lemma \ref{lem:N124} the coset stabilizer \(\mathcal{N}^{(1,2,4)}\) has just two orbits
\begin{align*}
	&\{\{1,3,2,24,4,20,19,13,29,30,21,32,16,25,26,31\}, \\
	&\{5,11,22,8,28,18,15,6,17,14,7,27,9,10,23,12\}\}
\end{align*}
We pick orbit representatives \(4,5\). The only possible new double coset is then \([t_1t_2t_4t_5]\) but we will see it is equal to \([t_1t_3]\). Start by conjugating Relation (\ref{reln:given-3}) by \(x^{-4}y^2\) and \(y^{-2}x^{-1}y^{-1}x\) to get
\begin{align*}
	(yx^{-1}yxy^{-1})^3t_1 &= t_1t_5 \\
	(xyx^{-1})^3t_6 &=t_6t_{19}.
\end{align*}
Then conjugate Relation (\ref{reln:182329}) by \(yx^{-1}yx^{-1}y^{-2}x^{-1}\), \(x^{-2}y^{-1}xy^{-1}x^{-2}\), and \(y^2x^{-1}yx^{-1}y^{-1}x^{-1}\) we get the relations
\begin{align*}
	yxy^{-1}x^{-2}y^{-2}t_7t_{19} &= t_{10}t_1 \\
	yxy^{-1}xy^{-1}x^{-1}yt_{21}t_{23} &= t_{30}t_7 \\
	y^{-1}xyx^{-1}yxy t_{25}t_6 &= t_{21}t_{23}.
\end{align*}
We compute
\begin{align*}
	(t_1t_2)t_4t_5 &\sim t_1t_2(t_1t_4)t_5 \sim t_1t_2(x^5y^3)^{-1}t_1t_5 \\
	&\sim t_4t_3(t_1t_5) \sim t_4t_3(yx^{-1}yxy^{-1})^3t_1 \\
	&\sim t_8(t_{10}t_1) \sim t_8 yxy^{-1}x^{-2}y^{-2}t_7t_{19} \\
	&\sim (t_{30}t_7)t_{19} \sim (t_{21}t_{23})t_{19} \\
	&\sim t_{25}(t_6t_{19}) \sim t_{25}(xyx^{-1})^3t_6 \\
	&\sim t_8t_6 \sim (t_1t_3)^{y^{-1}x^3yx}.
\end{align*}
Thus \([t_1t_2t_4t_5] = [t_1t_3]\).
This completes the double coset enumeration process.

\subsection{The Cayley Graph}

The Cayley Graph in Figure \ref{fig:cayley-graph} is a compact illustration of the double coset enumeration process. Each node is labelled by the corresponding double coset. The number inside the node is the number of single cosets determined by right multiplication by \(\mathcal{N}\). The edges indicate multiplication by the \(t_i\)s which have valency equal to the orbit size of the corresponding \(t_i\) under the coset stabilizer. 

\begin{figure}[h!]
\begin{tikzpicture}
\node[draw, circle, label=below:{\([\star]\)}]	at (0,0) (a)	{1};
\node[draw, circle, label=below:{\([t_1]\)}]	 at (2,0) (b)	{32};
\node[draw, circle, label=below:{\([t_1t_2]\)}] at (4,1) (c)	{32};
\node[draw, circle, label=below:{\([t_1t_2t_4]\)}] at (6,1) (d) {10};
\node[draw, circle, label=below:{\([t_1t_3]\)}] at (6,-1) (e) {160};
\node[draw, circle, label=below:{\([t_1t_8]\)}] at (4,-2) (f) {80};

\draw (a) -- (b) node[above, near end]{1} node[above, near start]{32};
\draw (b) -- (c) node[above, near end]{1} node[above, near start]{1};
\draw (b) -- (f) node[above, near end]{8} node[above, near start]{20};
\draw (b) -- (e) node[above, near end]{1} node[above, near start]{25};
\draw (c) -- (e) node[above, near end]{4} node[above, near start]{1};
\draw (e) -- (f) node[above, near end]{8} node[above, near start]{16};
\draw (c) -- (d) node[above, near end]{16} node[above, near start]{5};
\path (b) edge [out=75,in=105,looseness=10] node[above] {5} (b);
\path (c) edge [out=75,in=105,looseness=10] node[above] {1} (c);
\path (f) edge [out=165,in=195,looseness=10] node[above] {8} (f);
\path (e) edge [out=10,in=340,looseness=10] node[above, right] {17} (e);
\path (e) edge [bend right=60] node[right, near start]{1}  node[right, near end]{16} (d);
\end{tikzpicture}
\label{fig:cayley-graph}
\caption{They Cayley Graph}
\end{figure}

\section{The isomorphism \(\mathcal{G}\cong J_2\)}\label{sec:main}

\subsection{Simplicity of \(\mathcal{G}\)}

\begin{lemma}
	The order of \(\mathcal{G}\) is precisely \(604800\).
	\label{lem:order}
\end{lemma}

\begin{proof}
	By Proposition \ref{prop:sym-gen-exist}, the order of \(\mathcal{G}\) is bounded below by \(604800\). The double coset enumeration process in Section \ref{sec:dce} shows that the order of \(\mathcal{G}\) is bounded above by \(604800\). Hence, it is exactly \(604800\).
\end{proof}

\begin{lemma}
	The group \(\mathcal{G}\) is perfect.
	\label{lem:perfect}
\end{lemma}

\begin{proof}
	It suffices to show \(2^{1+4}:A_5\leq \mathcal{G}'=[\mathcal{G},\mathcal{G}]\) and \(t_1\in\mathcal{G}'\). It is well known that \( x^5\) generates the center of \(2^{1+4}:A_5\) and the quotient by \( (x^5)\) is perfect. By Relation \ref{reln:given-3} we have \([t_1,t_4] = \tau t_4\). Hence, \(t_4\in\mathcal{G}'\). Since \(\mathcal{G}\) is transitive we have \(t_1\in\mathcal{G}\). By Relation \ref{reln:12121} we have \([t_1,t_2]t_1 = x^5\in \mathcal{G}'\). Thus \(\mathcal{G}' = \mathcal{G}\).
\end{proof}

Let us consider the normal Abelian subgroup of the stabilizer of \([\star]\) given by \(\mathcal{H}=2^{1+4}\).

\begin{lemma}
	The conjugates of \(\mathcal{H}\) under the action of \(\mathcal{G}\) generate \(\mathcal{G}\).
	\label{lem:generates}
\end{lemma}

\begin{proof}
	Let \(\mathcal{K}\) be the subgroup of \(\mathcal{G}\) generated by the conjugates of \(\mathcal{H}\). By Relation \ref{reln:12121}, we have \(t_1t_2t_1t_2t_1=x^5\in\mathcal{H}\) so that by conjugating by \(t_1t_2\) yields \(t_1\in\mathcal{K}\). Since the action of \(2^{1+4}:A_5\) on \(X\) is transitive we have that all of the symmetric generators are in \(\mathcal{K}\). It follows that we have \(A_5\) is a subgroup of \(\mathcal{K}\) as well. This completes the proof.
\end{proof}

\begin{lemma}
	The group \(\mathcal{G}\) acts primitively on the set of single cosets \(\Omega\).
\end{lemma}

\begin{proof}
	Let \(\mathcal{B}\) be a non-trivial block in \(\Omega\). By multiplying by \(t_i\)'s we can assume that \(\mathcal{N}\) is in \(\Omega\). Thus the action of \(\mathcal{N}\) on \(\Omega\) preserves \(\mathcal{B}\). Since \(\mathcal{B}\) is nontrivial there is a coset of the form \(\mathcal{N}\omega\) in \(\mathcal{B}\) with \(\mathcal{N}\omega\neq \mathcal{N}\). 
	
	Suppose the length of \(\omega\) is less than 3. Then since \(\mathcal{B}\) is stabilized by \(\mathcal{N}\) it must be that \(\mathcal{B}\) has all double cosets in that orbit. Each of the double cosets is stabilized by a \(t_i\) so that \(\mathcal{B}\) is stabilized by \(t_i\) and thus it must be trivial. Finally, if \(\omega\) is length 3 then there must be another double coset of length two or one inside of \(\mathcal{B}\) as the size of the block must divide the order of the group.
\end{proof}

\begin{theorem}
	There is an isomorphism \(\mathcal{G}\cong J_2\).
	\label{thm:main}
\end{theorem}

\begin{proof}
	By Iwasawa's simplicity criterion and Lemmas \ref{lem:perfect} and \ref{lem:generates} we know that \(\mathcal{G}\) is simple. By Lemma \ref{lem:order} we know that \(\mathcal{G}\) is of order \(604800\). Since \(J_2\) is the unique simple group of that order \cite{hall-wales-69,ATLAS}, we have \(\mathcal{G}\cong J_2\).
\end{proof}

\subsection{Maximal Subgroups of \(\mathcal{G}\)}


There are nine conjugacy classes of maximal subgroups of \(J_2\), \cite{ATLAS}. In Table \ref{table:max-subgroups}, we give the corresponding generators for the maximal subgroups \(\mathcal{M}\) in \(\mathcal{G}\) in terms of our presentation.

\begin{table}[h!]
\label{table:max-subgroups}
\caption{Maximal Subgroups of \(\mathcal{G}\)}
\begin{tabular}{|c|c|} \hline
\(\mathcal{M}\) & Generators \\ \hline
\(U_3(3)\) & \(t^{xy^5(x^8)^{y^5x^9}}, y^t\) \\ \hline
\(3.A_6.2\) &  \(y^2  x^9  y  t  y  t  y^5  x  y^2,x  y^5  x  y^5  t  x  t  y  t  y  x\) \\ \hline
\(2^{1+4}:A_5\) & \(x,y\) \\ \hline
\(2^{2+4}:(3\times S_3) \) & \( t^{y^3x^9y^4x^9}, tyxyt\) \\ \hline
\(A_4\times A_5\) & \(y  x^9  y^2  t  y  t  x  y^5,x^2  y  t  x^{10}  y^5  t  x^8  y  x^8,
x^9  y  x^9  y  x^9  t  y  x  t  y^5,x^2  y^5  x  t  y^5  t  x  t  y  t
\) \\ \hline
\(A_5\times D_{10}\) & \(y^2,t,x^t\)\\ \hline
\(L_3(2):2\) & \( y^2  x^9  y  x^9  t  x^9  y  x^9  y^5  t  y^5,
t  y  x^2  y  t  y^5  x  y^5  x^8
\)\\ \hline
\(5^2:D_{12}\) & \(t^{y^5 x^9 y^5 t  x^9},
x  y  t  x  y  x^8  y  x  t  x  t  y^5,
x  t  y  x^9  y^5  t  x^9  y^5  x^9  t
\) \\ \hline
\(A_5\) & \(ytx(y^5x^8)^2y^{tx^4}, x^{y^5}t^{y^4}xt^{y^5x}t\) \\ \hline
\end{tabular}
\end{table}

\subsection{Connection with Janko's Conjecture}

Janko originally conjectured the existence of a simple group with involution centralizer of type \(2^{1+4}:A_5\). We have established \(\mathcal{G}\) is simple in the proof of Theorem \ref{thm:main}. 

\begin{theorem}
	The involution \(x^5\) is centralized by \(2^{1+4}:A_5\).
\end{theorem}

\begin{proof}
	By Lemma \ref{lem:order}, the coset stabilizers are not larger than what was computed in the double coset enumeration process. The centralizer of \(x^5\) certainly contains \(2^{1+4}:A_5\). Suppose there is another element \(\pi\omega\) that centralizes \(x^5\) where \(\omega\) is a non-empty word of length at most 3 in the \(t_i\)s. So we have
	\[
		x^5\pi\omega = \pi\omega x^5 = x^5\pi\omega^{x^5},
	\]
	as \(\pi\in 2^{1+4}:A_5\), which implies \(\omega = \omega^{x^5}\). There does not exist non-empty words of length at most three such that this happens as \(x^5\) doesn't stabilize any of the 32 letters of \(X\).
\end{proof}

\bibliographystyle{amsalpha}
\bibliography{j2.bib}

\end{document}